\documentclass[11pt]{article}
\usepackage{amsmath, amsthm, amscd, amsfonts, amssymb, graphicx,enumerate,color}

\theoremstyle{definition}
\newtheorem{theorem}{Theorem}[section]

\newtheorem{corollary}[theorem]{Corollary}
\newtheorem{proposition}[theorem]{Proposition}
\newtheorem{remark}[theorem]{Remark}

\begin{document}
\title{Refined inequalities on the weighted logarithmic mean}
\author{Shigeru Furuichi$^1$\footnote{E-mail:furuichi@chs.nihon-u.ac.jp} and Nicu\c{s}or Minculete$^2$\footnote{E-mail:minculeten@yahoo.com}\\
$^1${\small Department of Information Science,}\\
{\small College of Humanities and Sciences, Nihon University,}\\
{\small 3-25-40, Sakurajyousui, Setagaya-ku, Tokyo, 156-8550, Japan}\\
$^2${\small Transilvania University of Bra\c{s}ov, Bra\c{s}ov, 500091, Rom{a}nia}}
\date{}
\maketitle
{\bf Abstract.} Inspired by the recent  work by R.Pal et al., we give further refined inequalities for a convex Riemann integrable function, applying the standard Hermite-Hadamard inequality. Our approach is different from their one in \cite{PSMA2016}.
As corollaries, we give the refined inequalities on the weighted logarithmic mean and weighted identric mean.  
Some further extensions are also given.
\vspace{3mm}

\vspace{3mm}

\vspace{3mm}

%%%%%%%%%%%%%%%%%%%%%%%%%%%%%%%%%%%%%%%%%%%%%%%%%%%%%%%%%%%%%%%%%%%%%%%%%%%%
%%%%%%%%%%%%%%%%%%%%%%%%%%%%%%%%%%%%%%%%%%%%%%%%%%%%%%%%%%%%%%%%%%%%%%%%%%%%
%%%%%%%%%%%%%%%%%%%%%%%%%%%%%%%%%%%%%%%%%%%%%%%%%%%%%%%%%%%%%%%%%%%%%%%%%%%%
%%%%%%%%%%%%%%%%%%%%%%%%%%%%%%%%%%%%%%%%%%%  Section1  %%%%%%%%%%%%%%%%%%%%%%%
%%%%%%%%%%%%%%%%%%%%%%%%%%%%%%%%%%%%%%%%%%%%%%%%%%%%%%%%%%%%%%%%%%%%%%%%%%%%
%%%%%%%%%%%%%%%%%%%%%%%%%%%%%%%%%%%%%%%%%%%%%%%%%%%%%%%%%%%%%%%%%%%%%%%%%%%%
%%%%%%%%%%%%%%%%%%%%%%%%%%%%%%%%%%%%%%%%%%%%%%%%%%%%%%%%%%%%%%%%%%%%%%%%%%%%

\section{Introduction}
The inequalities on means attract many mathematicians for their depelopments. See \cite{FM2020} and references therein for example.
Recently, in  \cite[Theorem 2.2]{PSMA2016}, the weighted logarithmic mean was properly introduced and 
the inequalities among weighted means were shown as
\begin{equation}\label{sec1_eq01}
a\sharp_v b \leq L_v(a,b) \leq a\nabla_v b,
\end{equation}
where the weighted geometric mean is defined by $a \sharp_v b := a^{1-v}b^v $, the weighted arithmetic mean by
$a\nabla_v b :=(1-v)a+v b$ and the weighted logarithmic mean by \cite{PSMA2016}:
\begin{equation}\label{sec1_eq02}
L_v(a,b) := \frac{1}{\log a-\log b}\left(\frac{1-v}{v}(a-a^{1-v}b^v)+\frac{v}{1-v}(a^{1-v}b^v-b)\right)
\end{equation}
for $a,b >0$ and $v \in (0,1)$.
We easily find that $L_{1/2}(a,b)=\dfrac{a-b}{\log a-\log b}$ ($a\neq b$), with $L_{1/2}(a,a):=a$. This is the so-called logarithmic mean. We also find that $\lim\limits_{v\to 0}L_v(a,b) = a$ and $\lim\limits_{v\to 1}L_v(a,b) = b$.  Thus the inequalities given in \eqref{sec1_eq01} recover the well-known relations:
$$
\sqrt{ab}\leq \frac{a-b}{\log a-\log b} \leq \frac{a+b}{2}\,\,(a,b>0).
$$
We use the  symbols $\nabla$ and $\sharp$ simply, instead of $\nabla_{1/2}$ and $\sharp_{1/2}$.

R.Pal et al. obtained the inequalities given in \eqref{sec1_eq01} by their general result given in 
\cite[Theorem 2.1]{PSMA2016} which can be regarded as the generalization of the famous Hermite-Hadamard inequality with weight $v\in[0,1]$:
 \begin{equation}\label{sec1_eq03}
 f(a\nabla_v b)\leq C_{f,v}(a,b) \leq f(a)\nabla_vf(b)
 \end{equation}
 where
  \begin{equation}\label{sec1_eq04}
 C_{f,v}(a,b) := \left(\int_0^1 f\left(a\nabla_{vt} b\right)dt\right) \nabla_v \left(\int_0^1f\left((1-v)(b-a)t+a\nabla_v b\right)dt\right)
  \end{equation}
 for a convex Riemann integrable function, $a,b>0$ and $v\in [0,1]$. By elementary calculations, we find that the inequalities given in \eqref{sec1_eq03} recover the standard Hermite-Hadamard inequalities:
   \begin{equation}\label{sec1_eq05}
   f\left(\frac{a+b}{2}\right)\leq \frac{1}{b-a}\int_a^bf(t)dt \leq \frac{f(a)+f(b)}{2}.
   \end{equation}
   
 In this paper, we give a refinement of the ineqaulities given in \eqref{sec1_eq03} and as its consequence, we imply refined inequalities on the weighted logarithmic mean.
 
%%%%%%%%%%%%%%%%%%%%%%%%%%%%%%%%%%%%%%%%%%%%%%%%%%%%%%%%%%%%%%%%%%%%%%%%%
%%%%%%%%%%%%%%%%%%%%%%%%%%%%%%%%%%%%%%%%%%%%%%%%%%%%%%%%%%%%%%%%%%%%%%%%%
%%%%%%%%%%%%%%%%%%%%%%%%%%%%%%%%%%%%%%%%%%%%%%%%%%%%%%%%%%%%%%%%%%%%%%%%%
%%%%%%%%%%%%%%%%%%%%%%%%%Section 2%%%%%%%%%%%%%%%%%%%%%%%%%%%%%%%%%%%%%%%%
%%%%%%%%%%%%%%%%%%%%%%%%%%%%%%%%%%%%%%%%%%%%%%%%%%%%%%%%%%%%%%%%%%%%%%%%%
%%%%%%%%%%%%%%%%%%%%%%%%%%%%%%%%%%%%%%%%%%%%%%%%%%%%%%%%%%%%%%%%%%%%%%%%%
%%%%%%%%%%%%%%%%%%%%%%%%%%%%%%%%%%%%%%%%%%%%%%%%%%%%%%%%%%%%%%%%%%%%%%%%%
\section{Main results}
We firstly give the refined inequalities for \eqref{sec1_eq03} by repeating use of  the standard Hermite-Hadamard inequalities given in \eqref{sec1_eq05}.
\begin{theorem}\label{sec2_theorem01}
For every convex Riemann integrable function $f:[a,b]\to \mathbb{R}$ and $v\in[0,1]$, we have
\begin{equation}\label{sec2_eq01}
f\left(a\nabla_vb\right) \leq R^{(1)}_{f,v} (a,b) \leq C_{f,v}(a,b) \leq R^{(2)}_{f,v} (a,b) \leq f(a)\nabla_v f(b),
\end{equation}
where
\begin{equation}\label{sec2_eq02}
R^{(1)}_{f,v} (a,b) := f(a\nabla_{\frac{v}{2}}b)\nabla_vf(a\nabla_{\frac{1+v}{2}}b)
\end{equation}
and
\begin{equation}\label{sec2_eq03}
R^{(2)}_{f,v} (a,b) := \left(f(a)\nabla_v f(b)\right)\nabla\left( f(a \nabla_v b) \right).
\end{equation}
\end{theorem}

\begin{proof}
Applying  the standard Hermite-Hadamard inequalities \eqref{sec1_eq05} on the two intervals
$[a,(1-v)a+vb]$ and $[(1-v)a+vb,b]$, we obtain respectively
\begin{equation}\label{sec2_eq04}
f\left(\frac{(2-v)a+vb}{2}\right)\leq\frac{1}{v(b-a)}\int_a^{(1-v)a+vb}f(t)dt\leq \frac{f(a)+f((1-v)a+vb)}{2}
\end{equation}
and
\begin{equation}\label{sec2_eq05}
f\left(\frac{(1-v)a+(1+v)b}{2}\right) \leq \frac{1}{(1-v)(b-a)}\int_{(1-v)a+vb}^bf(t)dt\leq \frac{f(b)+f((1-v)a+vb)}{2}.
\end{equation}
Multiplying both sides in \eqref{sec2_eq04} and \eqref{sec2_eq05} by $(1-v)$ and $v$ respectively and summing each side,  we obtain
\begin{equation}\label{sec2_eq06}
R^{(1)}_{f,v} (a,b)  \leq \frac{1-v}{v(b-a)}\int_a^{(1-v)a+vb}f(t)dt + \frac{v}{(1-v)(b-a)}\int_{(1-v)a+vb}^bf(t)dt \leq R^{(2)}_{f,v} (a,b),
\end{equation}
which is equivalent to 
\begin{equation}\label{sec2_eq07}
R^{(1)}_{f,v} (a,b)  \leq C_{f,v}(a,b) \leq R^{(2)}_{f,v} (a,b),
\end{equation}
by replacing the variables such as $t:=v(b-a)s+a$ in the first term and $t:=(1-v)(b-a)u+(1-v)a+vb$ in the second term of the integral parts in \eqref{sec2_eq06}.

Finally we estimate $R^{(1)}_{f,v} (a,b)$ and $R^{(2)}_{f,v} (a,b)$. Since the function $f$ is convex,
we have
$$
R^{(1)}_{f,v} (a,b) \geq f\left(\frac{\left((1-v)(2-v)+v(1-v)\right)a+\left(v(1-v)+v(1+v)\right)b}{2}\right) =f(a\nabla_vb)
$$   
and
$$
R^{(2)}_{f,v} (a,b) \leq \left(f(a)\nabla_v f(b)\right)\nabla\left(f(a)\nabla_v f(b)\right) =f(a)\nabla_v f(b).
$$
Thus we complete the proof.
\end{proof}

\begin{corollary}\label{sec2_corollary01}
For $a, b >0$ and $v\in (0,1)$, we have
\begin{equation}\label{sec2_eq08}
a\sharp_vb \leq \left(a\sharp_{\frac{v}{2}} b\right)\nabla_v \left(a \sharp_{\frac{1+v}{2}}b\right)\leq L_v(a,b)\leq \left(a\nabla_v b\right) \nabla\left(a\sharp_vb\right) \leq a \nabla_vb.
\end{equation}
\end{corollary}
\begin{proof}
Applying the convex function $f(t):=e^t$ in Theorem \ref{sec2_theorem01}, we have for $b \geq a >0$
\begin{eqnarray*}
&&e^{(1-v)a+vb} \leq (1-v)e^{\frac{(2-v)a+vb}{2}}+ve^{\frac{(1-v)a+(1+v)b}{2}}\leq (1-v)\int_0^1 e^{v(b-a)t+a}dt\\
&&+v\int_0^1 e^{(1-v)(b-a)t+(1-v)a+vb}dt \leq \frac{(1-v)e^a+ve^b+e^{(1-v)a+vb}}{2} \leq(1-v)e^a+ve^b.
\end{eqnarray*}
By elementary calculations, we have
\begin{eqnarray*}
&&(1-v)\int_0^1 e^{v(b-a)t+a}dt+v\int_0^1 e^{(1-v)(b-a)t+(1-v)a+vb}dt\\
&&= \frac{1-v}{v(b-a)}\left(e^{(1-v)a+vb}-e^a\right)+\frac{v}{(1-v)(b-a)}\left(e^b-e^{(1-v)a+vb}\right).
\end{eqnarray*}
Replacing $e^a$ and $e^b$ with $a$ and $b$ respectively, we obtain
the inequalities \eqref{sec2_eq08} for $b \geq a >0$ and $v\in (0,1)$.
Dividing both sides of the inequalities \eqref{sec2_eq08} by $a$ and putting $\frac{b}{a}:=t \geq 1$, we have
\begin{equation}\label{sec2_proof_theorem01_eq01}
t^{v} \leq (1-v)t^{\frac{v}{2}} +v t^{\frac{1+v}{2}}\leq L_v(1,t)\leq \frac{1}{2}\left((1-v)+vt + t^v\right) \leq (1-v)+vt,\,\,(t \geq 1,\,\,v\in(0,1)).
\end{equation}
Putting $s:=\frac{1}{t} \leq 1$ and $u:=1-v$, and then multiplying
 both sides by  $s>0$, we have
\begin{equation}\label{sec2_proof_theorem01_eq02}
s^{u} \leq (1-u)s^{\frac{u}{2}} +u s^{\frac{1+u}{2}}\leq L_u(1,s)\leq \frac{1}{2}\left((1-u)+us + s^u\right) \leq (1-u)+us\,\,(0< s \leq 1,\,\,u\in(0,1))
\end{equation}
by elementary calculations. Thus we have the inequalities: 
\begin{equation}\label{sec2_proof_theorem01_eq03}
t^{v} \leq (1-v)t^{\frac{v}{2}} +v t^{\frac{1+v}{2}}\leq L_v(1,t)\leq \frac{1}{2}\left((1-v)+vt + t^v\right) \leq (1-v)+vt \,\,
(t>0,\,\,v \in (0,1)).
\end{equation}
Therefore we complete the proof by putting $t:=\frac{b}{a}$ for any $a,b>0$ in \eqref{sec2_proof_theorem01_eq03} and then multiplying both sides by $a>0$.
\end{proof}

We note that the third and fourth inequalities have already been  given in \cite[Lemma 2.3]{PSMA2016}.
However, the first and second inequalities are new results. In addition, our approaches are different from the author's in \cite{PSMA2016}.

We give the inequalities on the weighted identric mean which was defined in \cite{PSMA2016}
as
\begin{equation}\label{sec2_eq10}
 I_v(a,b) := \frac{1}{e}\left(a\nabla_v b\right)^{\frac{(1-2v)\left(a\nabla_v b\right)}{v(1-v)(b-a)}}\left(\frac{b^{\frac{vb}{1-v}}}{a^{\frac{(1-v)a}{v}}}\right)^{\frac{1}{b-a}},\quad v\in(0,1).
\end{equation}
It is easy to check that $I_{1/2}(a,b)$ recovers the usual identric mean $I(a,b):=\frac{1}{e}\left(\frac{b^b}{a^a}\right)^{\frac{1}{b-a}}$,
with $\lim\limits_{v\to 0}I_v(a,b) = a$ and $\lim\limits_{v\to 1}I_v(a,b) = b$.

\begin{corollary}\label{sec2_corollary02}
For $a,b>0$ and $v\in (0,1)$, we have
\begin{equation}\label{sec2_eq09}
a\sharp_v b \leq \left(a\sharp_v b\right) \sharp\left(a \nabla_v b\right) \leq I_v(a,b) \leq 
\left(a\nabla_{\frac{v}{2}}b\right)\sharp_v \left(a\nabla_{\frac{1+v}{2}}b\right)\leq a \nabla_v b.
\end{equation}
\end{corollary}
\begin{proof}
Applying the convex function $f(t):=-\log t$, $(t>0)$ in Theorem \ref{sec2_theorem01}, we have for $b \geq a >0$ with elementary calculations 
\begin{eqnarray*}
&&\log a^{1-v} b^v \leq \log \left(a^{\frac{1-v}{2}}b^{\frac{v}{2}} \left((1-v)a+vb\right)^{\frac{1}{2}}\right) \\
&& \leq \frac{1-v}{v(b-a)}\left\{ \left((1-v)a+vb\right)\log \left((1-v)a+vb\right)-\left((1-v)a+vb\right)-a\log a+a  \right\}\\
&&+ \frac{v}{(1-v)(b-a)}\left\{ b\log b - b-\left((1-v)a+vb\right)\log \left((1-v)a+vb\right) +\left((1-v)a+vb\right)\right\}\\
&&\leq \log\left(\left(1-\frac{v}{2}\right)a+\frac{v}{2}b \right)^{1-v}\left(\left(1-\frac{1+v}{2}\right)a+\frac{1+v}{2}b\right)^{v} \leq \log\left((1-v)a+vb\right).
\end{eqnarray*}
We calculate the following.
\begin{eqnarray*}
&&\frac{1-v}{v(b-a)}\left\{ \left((1-v)a+vb\right)\log \left((1-v)a+vb\right)-\left((1-v)a+vb\right)-a\log a+a  \right\}\\
&&+ \frac{v}{(1-v)(b-a)}\left\{ b\log b - b-\left((1-v)a+vb\right)\log \left((1-v)a+vb\right) +\left((1-v)a+vb\right)\right\}\\
&&=\log \left\{(1-v)a+vb\right\}^{\frac{(1-2v)\left\{(1-v)a+vb\right\} }{v(1-v)(b-a)}}b^{\frac{vb}{(1-v)(b-a)}}a^{-\frac{(1-v)a}{v(b-a)}}-1\\
&&= \log \frac{1}{e} \left\{(1-v)a+vb\right\}^{\frac{(1-2v)\left\{(1-v)a+vb\right\} }{v(1-v)(b-a)}}\left(\frac{b^{\frac{vb}{1-v}}}{a^{\frac{(1-v)a}{v}}}\right)^{\frac{1}{b-a}}.
\end{eqnarray*}
Thus we complete the proof for any $a,b >0$ in a similar way to the proof of Corollary \ref{sec2_corollary01}.
\end{proof}

Our Corollary \ref{sec2_corollary02} clearly refines \cite[Theorem 3.1]{PSMA2016}.

According to the inequalities shown in \cite[Theorem 3.3]{Mit2010} for a convex function $f$,
\begin{equation}\label{sec2_eq11}
 2v_{\min}\cdot \Delta_{f,1/2}(a,b)  
 \leq \Delta_{f,v}(a,b)
 \leq 2 v_{\max}\cdot \Delta_{f,1/2}(a,b)
 \end{equation}
where $v_{\min}:=\min\left\{1-v,v\right\}$, $v_{\max}:=\max\left\{1-v,v\right\}$ and $v\in [0,1]$
\begin{equation}\label{mit_ineq}
\Delta_{f,v}(a,b):= f(a)\nabla_vf(b)-f\left(a\nabla_vb\right)\geq 0,
 \end{equation}
we obtain the further refinements of Theorem \ref{sec2_theorem01}.

\begin{proposition}
Under the same assumption as in Theorem \ref{sec2_theorem01}, we have
\begin{equation}\label{sec2_eq17}
f\left(a\nabla_vb\right) \leq Q^{(1)}_{f,v} (a,b) \leq R^{(1)}_{f,v} (a,b) \leq C_{f,v}(a,b) \leq R^{(2)}_{f,v} (a,b) \leq Q^{(2)}_{f,v} (a,b)\leq f(a)\nabla_v f(b),
\end{equation}
where
$$
Q^{(1)}_{f,v} (a,b) := f(a\nabla_vb)+2v_{\min}\cdot \Delta_{f,1/2}\left(a\nabla_{\frac{v}{2}}b,a\nabla_{\frac{1+v}{2}}b\right)
$$
and
$$
Q^{(2)}_{f,v} (a,b) :=  f(a)\nabla_vf(b)-v_{\min}\cdot \Delta_{f,1/2}\left(a,b\right).
$$
\end{proposition}
\begin{proof}
Using the first inequality from relation (\ref{sec2_eq11}) and replacing $a$ and $b$ by $a\nabla_{\frac{v}{2}}b$ and $a\nabla_{\frac{1+v}{2}}b$ respectively, we deduce
\begin{equation*}
 2v_{\min}\cdot \Delta_{f,1/2}(a\nabla_{\frac{v}{2}}b, a\nabla_{\frac{1+v}{2}}b)  
 \leq \Delta_{f,v}(a\nabla_{\frac{v}{2}}b, a\nabla_{\frac{1+v}{2}}b)
 \end{equation*}
\begin{equation*}
 =R^{(1)}_{f,v} (a,b)-f\left((a\nabla_{\frac{v}{2}}b)\nabla_{\frac{v}{2}} (a\nabla_{\frac{1+v}{2}}b)\right)=R^{(1)}_{f,v} (a,b)-f\left(a\nabla_{v}b\right).
 \end{equation*}
Using the first inequality in \eqref{sec2_eq11} again, we have
\begin{eqnarray*}
&& R^{(2)}_{f,v} (a,b) = \left(f(a)\nabla_v f(b)\right)\nabla\left( f(a \nabla_v b) \right)=
\frac{1}{2}\left\{f(a)\nabla_vf(b)+f\left(a\nabla_vb\right)\right\}\leq \\
&& f(a)\nabla_vf(b)-v_{\min}\cdot \Delta_{f,1/2}\left(a,b\right)=Q^{(2)}_{f,v} (a,b)\leq f(a)\nabla_vf(b).
\end{eqnarray*}
\end{proof}

\begin{remark}
\begin{itemize}
\item[(i)] From the inequality $Q_{f,v}^{(2)}(a,b) \geq Q_{f,v}^{(1)}(a,b)$ in \eqref{sec2_eq17}, we find that
$$
\Delta_{f,v}(a,b) \geq v_{\min} \left(\Delta_{f,1/2}(a,b)+2\Delta_{f,1/2}\left(a\nabla_{\frac{v}{2}}b,a\nabla_{\frac{1+v}{2}}b\right)\right)\geq 0,
$$
which gives a refinement of \eqref{mit_ineq}.
\item[(ii)] From the second inequality of \eqref{sec2_eq11}, we also find that
$$R_{f,v}^{(1)}(a,b) \leq P_{f,v}^{(1)}(a,b),\quad P_{f,v}^{(2)}(a,b)\leq R_{f,v}^{(2)}(a,b)$$
where
$$
P_{f,v}^{(1)}(a,b):=f(a\nabla_v b)+2v_{\max}\cdot \Delta_{f,1/2}\left(a\nabla_{\frac{v}{2}}b,a\nabla_{\frac{1+v}{2}}b\right)
$$
and
$$
P_{f,v}^{(2)}(a,b) := f(a)\nabla_vf(b)-v_{\max}\cdot \Delta_{f,1/2}\left(a,b\right).
$$
However, there is no ordering between $ P_{f,v}^{(1)}(a,b)$ and  $P_{f,v}^{(2)}(a,b)$, since we have the following numerical examples.
$$ P_{\exp,1/4}^{(1)}(4,1)-P_{\exp,1/4}^{(2)}(4,1) \simeq 4.35403,\,\,
P_{\exp,1/4}^{(1)}(8,1)-P_{\exp,1/4}^{(2)}(8,1) \simeq -30.7996.
$$

\end{itemize}
\end{remark}
%%%%%%%%%%%%%%%%%%%%%%%%%%%%%%%%%%%%%%%%%%%%%%%%%%%%%%%%%%%%%
%%%%%%%%%%%%%%%%%%%%%%%%%%%%%%%%%%%%%%%%%%%%%%%%%%%%%%%%%%%%%
%%%%%%%%%%%%%%%%%%%%%%%%%%%%%%%%%%%%%%%%%%%%%%%%%%%%%%%%%%%%%
%%%%%%%%%%%%                Section 3              %%%%%%%%%%%%%%%%%%%%%%%
%%%%%%%%%%%%%%%%%%%%%%%%%%%%%%%%%%%%%%%%%%%%%%%%%%%%%%%%%%%%%
%%%%%%%%%%%%%%%%%%%%%%%%%%%%%%%%%%%%%%%%%%%%%%%%%%%%%%%%%%%%%
%%%%%%%%%%%%%%%%%%%%%%%%%%%%%%%%%%%%%%%%%%%%%%%%%%%%%%%%%%%%%

\section{Reverses and refinements by differentiable functions}
We extend the above results for differentiable functions.
From \cite{CeDra}, if $f:I\rightarrow \mathbb{R}$ is a differentiable function on $I^o$ (interior of $I$) and if $f'\in L[a,b]$(the space of Riemann integrable function on $[a,b]$), where $a,b\in I$ with $a < b$,  then the following equality holds for each $x\in [a,b]$:
\begin{equation}\label{sec2_eq18}
   f(x)-\frac{1}{b-a}\int_a^bf(t)dt =\frac{(x-a)^2}{b-a}\int_0^1vf'((1-v)a+vx)dv-\frac{(b-x)^2}{b-a}\int_0^1vf'((1-v)b+vx)dv.
\end{equation}
If  we choose $x=\dfrac{a+b}{2}$ in  \eqref{sec2_eq18}, then we have
\begin{eqnarray}\label{sec2_eq19}
   &&f\left(\frac{a+b}{2}\right)-\frac{1}{b-a}\int_a^bf(t)dt \nonumber \\
   &&=\frac{b-a}{4}\left\{\int_0^1vf'\left((1-v)a+v\frac{a+b}{2}\right)dv-\int_0^1vf'\left((1-v)b+v\frac{a+b}{2}\right)dv\right\}.
   \end{eqnarray}
 In \cite{DraAg} we found the relation
\begin{equation}\label{sec2_eq20}
   \frac{f(a)+f(b)}{2}-\frac{1}{b-a}\int_a^bf(t)dt =\frac{b-a}{2}\int_0^1(1-2v)f'(va+(1-v)b)dv.
   \end{equation}
Here, we have the equality:
\begin{equation*}
\int_0^1(1-2v)f'(va+(1-v)b)dv=\int_0^1(2v-1)f'((1-v)a+vb)dv=\frac{2}{(b-a)^2}\int_a^b\left(t-\frac{a+b}{2}\right)f'(t)dt.
   \end{equation*}
Thus we have the following equality from \eqref{sec2_eq20} with the equality
\begin{equation}\label{sec2_eq21}
\frac{f(a)+f(b)}{2}-\frac{1}{b-a}\int_a^bf(t)dt =\frac{1}{b-a}\int_a^b\left(t-\frac{a+b}{2}\right)f'(t)dt.
\end{equation}

\begin{theorem}\label{sec3_theorem02}
For every convex differentiable function $f:[a,b]\to \mathbb{R}$ with $f'\in L[a,b]$ and $|f'(x)|\leq K$, we have
\begin{equation}\label{sec3_theorem02_eq01}
 C_{f,v}(a,b)-R^{(1)}_{f,v} (a,b) \leq \frac{v(1-v)K(b-a)}{2}
\end{equation}
and
\begin{equation}\label{sec3_theorem02_eq02}
R^{(2)}_{f,v} (a,b)- C_{f,v}(a,b) \leq \frac{v(1-v)K(b-a)}{2}.
\end{equation}
\end{theorem}
\begin{proof} If $|f'(x)|\leq K$, then from \eqref{sec2_eq19} we deduce
   \begin{equation}\label{sec2_eq22}
   \frac{1}{b-a}\int_a^bf(t)dt-f\left(\frac{a+b}{2}\right)\leq\frac{K(b-a)}{4}
   \end{equation}
  and from \eqref{sec2_eq21} we obtain
  \begin{equation}\label{sec2_eq23}
\frac{f(a)+f(b)}{2}-\frac{1}{b-a}\int_a^bf(t)dt \leq\frac{K}{b-a}\int_a^b\left|t-\frac{a+b}{2}\right|dt=\frac{K(b-a)}{4}.
\end{equation}
We obtain \eqref{sec3_theorem02_eq01} by applying  the inequalities \eqref{sec2_eq22}  on the two intervals
$[a,(1-v)a+vb]$ and $[(1-v)a+vb,b]$, and then multiplying them by $(1-v)$ and $v$ and summing them.
In the same way with \eqref{sec2_eq23}, we obtain \eqref{sec3_theorem02_eq02}.
\end{proof}
 \begin{corollary}\label{sec3_corollary01}
For $b\geq a>0$ and $v\in (0,1)$, we have
\begin{equation}\label{sec2_eq24}
L_v(a,b)\leq  \left(a\sharp_{\frac{v}{2}} b\right)\nabla_v \left(a \sharp_{\frac{1+v}{2}}b\right)+\frac{v(1-v)b}{2}\log \frac{b}{a}
\end{equation}
and
\begin{equation}\label{sec2_eq25}
\left(a\nabla_v b\right) \nabla\left(a\sharp_vb\right)\leq L_v(a,b)+ \frac{v(1-v)b}{2}\log \frac{b}{a}.
\end{equation}
\end{corollary}
\begin{proof}
Applying the convex function $f(t):=e^t$ in Theorem \ref{sec3_theorem02}, we have the relations of the statement, since
we have
\begin{eqnarray*}
&&C_{\exp,v}(a,b)=L_v(e^a,e^b),\\
&&R_{\exp,v}^{(1)}(a,b)=\left(e^a\sharp_{\frac{v}{2}}e^b\right)\nabla_v\left(e^a\sharp_{\frac{1+v}{2}}e^b\right),\\
&&R_{\exp,v}^{(2)}(a,b)= \left(e^a\nabla_ve^b\right)\nabla\left(e^a\sharp_ve^b\right)
\end{eqnarray*}
and we can take $K=e^b$ for $t \in [a,b]$. Finally we replace $e^a$ and $e^b$ by $a$ and $b$, respectively. 
\end{proof}

The inequalities \eqref{sec2_eq24} and \eqref{sec2_eq25} give (difference type) reverses for the 2nd and 3rd inequalities in \eqref{sec2_eq08}, respectively.

 \begin{corollary}\label{sec3_corollary02}
For $b\geq a>0$ and $v\in (0,1)$, we have
\begin{equation}\label{sec2_eq26}
\left(a\nabla_{\frac{v}{2}} b\right)\sharp_v \left(a \nabla_{\frac{1+v}{2}}b\right) \leq
e^\frac{v(1-v)(b-a)}{2a} I_v(a,b) 
\end{equation}
and
\begin{equation}\label{sec2_eq27}
I_v(a,b) \leq 
e^\frac{v(1-v)(b-a)}{2a} \left(a\sharp_v b\right) \sharp\left(a\nabla_vb\right).
\end{equation}
\end{corollary}
\begin{proof}
Applying the convex function $f(t):=-\log t,\,\,(t>0)$ in Theorem \ref{sec3_theorem02}, we have the relations of the statement, since
we have
\begin{eqnarray*}
&&C_{-\log,v}(a,b)=-\log I_v(a,b),\\
&&R_{-\log,v}^{(1)}(a,b)= -\log \left(a\nabla_{\frac{v}{2}}b\right)\sharp_v\left(a\nabla_{\frac{1+v}{2}}b\right),\\
&&R_{-\log,v}^{(2)}(a,b)= -\log \left(a\sharp_v b\right)\sharp\left(a\nabla_v b\right)
\end{eqnarray*}
and we can take $K=\frac{1}{a}$ for $t \in [a,b]$.
\end{proof}
The inequalities \eqref{sec2_eq26} and \eqref{sec2_eq27} give (ratio type) reverses for the 3rd and 2nd inequalities in \eqref{sec2_eq09}, respectively.

   We extend the above results for the twice differentiable functions.
From \cite{DraCeSo1},\cite{DraCeSo2} and \cite{FaLaBa}, assume that  $f:I\rightarrow \mathbb{R}$ is a continuous on $I$, twice differentiable on $I^o$ and there exist $m=\underset{x\in I^o}{\inf}f"(x)$ and $M=\underset{x\in I^o}{\sup}f"(x)$, $a,b\in I$ with $a < b$, then the following inequalities hold:
\begin{equation}\label{sec2_eq28}
   \frac{m}{3}\left(\frac{b-a}{2}\right)^2\leq\frac{f(a)+f(b)}{2}-\frac{1}{b-a}\int_a^bf(t)dt\leq \frac{M}{3}\left(\frac{b-a}{2}\right)^2
   \end{equation}
 and
 \begin{equation}\label{sec2_eq29}
   \frac{m}{6}\left(\frac{b-a}{2}\right)^2\leq \frac{1}{b-a}\int_a^bf(t)dt-f\left(\frac{a+b}{2}\right)\leq \frac{M}{6}\left(\frac{b-a}{2}\right)^2.
   \end{equation}
\begin{theorem}\label{sec3_theorem03}
Assume that  $f:I\rightarrow \mathbb{R}$ is a continuous on $I$, twice differentiable on $I^o$ and there exist $m=\underset{x\in I^o}{\inf}f"(x)$ and $M=\underset{x\in I^o}{\sup}f"(x)$, $a,b\in I$ with $a < b$, we have
\begin{equation}\label{sec2_eq30}
 \frac{v(1-v)m}{6}\left(\frac{b-a}{2}\right)^2  \leq C_{f,v}(a,b)-R_{f,v}^{(1)}(a,b) \leq \frac{v(1-v)M}{6}\left(\frac{b-a}{2}\right)^2
\end{equation}
and
\begin{equation}\label{sec2_eq31}
 \frac{v(1-v)m}{3}\left(\frac{b-a}{2}\right)^2  \leq R_{f,v}^{(2)}(a,b)- C_{f,v}(a,b)\leq \frac{v(1-v)M}{3}\left(\frac{b-a}{2}\right)^2.
\end{equation}
\end{theorem}
\begin{proof} Applying  the inequality (\ref{sec2_eq28}) on the two intervals
$[a,(1-v)a+vb]$ and $[(1-v)a+vb,b]$, we obtain
\begin{equation}\label{sec2_eq32}
   \frac{m}{6}\left(\frac{v(b-a)}{2}\right)^2\leq \frac{1}{v(b-a)}\int_a^bf(t)dt-f\left(a\nabla_{\frac{v}{2}}b\right)\leq \frac{M}{6}\left(\frac{v(b-a)}{2}\right)^2
   \end{equation}
   and
   \begin{equation}\label{sec2_eq33}
   \frac{m}{6}\left(\frac{(1-v)(b-a)}{2}\right)^2\leq \frac{1}{(1-v)(b-a)}\int_a^bf(t)dt-f\left(a\nabla_{\frac{1+v}{2}}b\right)\leq \frac{M}{6}\left(\frac{(1-v)(b-a)}{2}\right)^2.
   \end{equation}
  Multiplying both sides in (\ref{sec2_eq32}) and (\ref{sec2_eq33}) by $(1-v)$ and $v$ respectively and summing each side,  we obtain the relations of the statement. Similarly, applying  the inequality (\ref{sec2_eq29}), we deduce the inequality (\ref{sec2_eq33}).
 \end{proof} 

 \begin{corollary}\label{sec3_corollary03}
For $b\geq a>0$ and $v\in (0,1)$, we have
\begin{equation}\label{sec2_eq34}
\frac{v(1-v)a}{24}\log^2\frac{b}{a}\leq L_v(a,b)-\left(a\sharp_{\frac{v}{2}} b\right)\nabla_v \left(a \sharp_{\frac{1+v}{2}}b\right)\leq  \frac{v(1-v)b}{24}\log^2\frac{b}{a}
\end{equation}
and
\begin{equation}\label{sec2_eq35}
\frac{v(1-v)a}{12}\log^2\frac{b}{a}\leq\left(a\nabla_v b\right) \nabla\left(a\sharp_vb\right)-L_v(a,b)\leq \frac{v(1-v)b}{12}\log^2\frac{b}{a}.
\end{equation}
\end{corollary}
\begin{proof}
Applying the convex function $f(t):=e^t$ in Theorem \ref{sec3_theorem03}, we have the relations of the statement, since $m=e^a$ and $M=e^b$. Finally we replace $e^a$ and $e^b$ by $a$ and $b$, respectively. 
\end{proof}

The inequalities \eqref{sec2_eq34} and \eqref{sec2_eq35}  give a better (difference type) refinement for the 2nd and 3rd inequality in \eqref{sec2_eq08}, respectively.

 \begin{corollary}\label{sec3_corollary04}
For $b\geq a>0$ and $v\in (0,1)$, we have
\begin{equation}\label{sec2_eq36}
e^\frac{-v(1-v)(b-a)^2}{24a^2}\left(a\nabla_{\frac{v}{2}} b\right)\sharp_v \left(a \nabla_{\frac{1+v}{2}}b\right)\leq  I_v(a,b)\leq e^\frac{-v(1-v)(b-a)^2}{24b^2}\left(a\nabla_{\frac{v}{2}} b\right)\sharp_v \left(a \nabla_{\frac{1+v}{2}}b\right)
\end{equation}
and
\begin{equation}\label{sec2_eq37}
e^\frac{v(1-v)(b-a)^2}{12b^2}\left(a\sharp_v b\right) \sharp\left(a\nabla_vb\right)\leq  I_v(a,b)\leq e^\frac{v(1-v)(b-a)^2}{12a^2}\left(a\sharp_v b\right) \sharp\left(a\nabla_vb\right).
\end{equation}
\end{corollary}
\begin{proof}
Applying the convex function $f(t):=-\log t,\,\,(t>0)$ in Theorem \ref{sec3_theorem03}, we have the relations of the statement, since 
$m=\frac{1}{b^2}$ and $M=\frac{1}{a^2}$.
\end{proof}
 
 The inequalities \eqref{sec2_eq36} and \eqref{sec2_eq37}  give a better (ratio type) refinement for the 3rd and 2nd inequality in \eqref{sec2_eq09}, respectively.
%%%%%%%%%%%%%%%%%%%%%%%%%%%%%%%%%%%%%%%%%%%%%%%%%%%%%%%%%%%%%
%%%%%%%%%%%%%%%%%%%%%%%%%%%%%%%%%%%%%%%%%%%%%%%%%%%%%%%%%%%%%
%%%%%%%%%%%%%%%%%%%%%%%%%%%%%%%%%%%%%%%%%%%%%%%%%%%%%%%%%%%%%
%%%%%%%%%%%%                Section 4              %%%%%%%%%%%%%%%%%%%%%%%
%%%%%%%%%%%%%%%%%%%%%%%%%%%%%%%%%%%%%%%%%%%%%%%%%%%%%%%%%%%%%
%%%%%%%%%%%%%%%%%%%%%%%%%%%%%%%%%%%%%%%%%%%%%%%%%%%%%%%%%%%%%
%%%%%%%%%%%%%%%%%%%%%%%%%%%%%%%%%%%%%%%%%%%%%%%%%%%%%%%%%%%%%

\section{Concluding remarks}
Our obtained results in this paper can be extended to the operator inequalities. We give operator inequalities corresponding to Corollary \ref{sec2_corollary01}. We omit the other cases.
For strictly positive operators $A$ and $B$,  the weighted geometric operator mean and arithmetic operator mean are defined as
$$
A\sharp_v B:=A^{1/2}\left(A^{-1/2}BA^{-1/2}\right)^vA^{1/2},\quad A\nabla_v B:=(1-v)A+v B.
$$
It is known that an operator mean $M(A,B)$ is associated with the representing function $f(t)=m(1,t)$ with a mean $m(a,b)$ for positive numbers $a,b$, in the following
$$
M(A,B)=A^{1/2}f\left(A^{-1/2}BA^{-1/2}\right)A^{1/2}
$$
in the general operator mean theory by Kubo-Ando \cite{KA1980}.
Thus it is understood that the weighted logarithmic operator mean $A \ell_v B$ is defined  through the representing function $L_v(1,t)$ for $v\in (0,1)$. 

From Corollary \ref{sec2_corollary01} and Kubo-Ando theory (or standard functional calculus), we can obtain the following operator inequalities. However, we state an alternative proof for the scalar inequalities on the representing functions.

\begin{theorem}\label{sec4_theorem4.1}
For any $v\in (0,1)$ and strictly positive operators $A$ and $B$, we have
$$
A\sharp_vB \leq  (1-v)A\sharp_{\frac{v}{2}}B+vA\sharp_{\frac{1+v}{2}}B\leq A\ell_vB \leq \frac{1}{2}\left(A\sharp_{v}B+A\nabla_{v}B\right)\leq  A\nabla_{v}B.
$$
\end{theorem}
\begin{proof}
It is sufficient to prove the following scalar inequalities:
\begin{equation}\label{sec4_proof_theorem4.1_eq01}
t^{v}\leq (1-v)t^{v/2}+vt^{(1+v)/2} \leq L_v(1,t) \leq \frac{1}{2}\left(t^v+(1-v)+vt\right)\leq (1-v) +vt
\end{equation}
where
$$
 L_v(1,t):= \frac{1}{\log t}\left(\frac{1-v}{v}\left(t^v-1\right)+\frac{v}{1-v}\left(t-t^v\right)\right)\,\,(t>0,\,\,v\in(0,1)).
$$
The fourth inequality in \eqref{sec4_proof_theorem4.1_eq01} is trivial and third one in \eqref{sec4_proof_theorem4.1_eq01} was proven in \cite[Lemma 2.3]{PSMA2016}.
The first inequality in \eqref{sec4_proof_theorem4.1_eq01} can be proven by the fact that the arithmetic mean is greater than or equal to the geometric mean as 
$(1-v)t^{v/2}+vt^{(1+v)/2} \geq t^{v(1-v)/2}t^{v(1+v)/2}=t^v$. The second inequality in  \eqref{sec4_proof_theorem4.1_eq01} can be proven by the use of the following first inequality:
\begin{equation}\label{sec4_proof_theorem4.1_eq02}
\frac{x^2-1}{\log x^2} \geq x >0.
\end{equation}
Putting $x:=t^{v/2}$ and $x:=t^{(v-1)/2}$ in \eqref{sec4_proof_theorem4.1_eq02}, we have respectively
$$
t^{v/2}\leq \frac{t^v-1}{v \log t}\,\,\,\,\text{and}\,\,\,\,t^{(v-1)/2}\leq \frac{t^{v-1}-1}{(v-1)\log t}\Leftrightarrow t^{(1+v)/2}\leq \frac{t-t^v}{(1-v)\log t}.
$$
Multiplying  the first and second inequality in the above by $(1-v)$ and $v$ and then summing them, we obtain  the second inequality in  \eqref{sec4_proof_theorem4.1_eq01}. Finally, replacing $t$ by $A^{-1/2}BA^{-1/2}$ in the inequalities \eqref{sec4_proof_theorem4.1_eq01} and then multiplying  both sides by $A^{1/2}$, we complete the proof.
\end{proof}

The upper bound of $A\ell_vB$ has already given in \cite[Theorem 2.4]{PSMA2016}. But the lower bound of  $A\ell_vB$ is a new result in Theorem \ref{sec4_theorem4.1}.
%It is also notable that Theorem \ref{sec4_theorem4.1} (namely the inequalities \eqref{sec4_proof_theorem4.1_eq01}) assures the validality of Corollary \ref{sec2_corollary01} for all $a,b>0$.
%We omit the studies for the other corollaries.

%\section*{Acknowledgement}
\section*{Acknowledgements}
The authors would like to thank the referees for their careful and insightful comments to improve our manuscript.
The author (S.F.) was partially supported by JSPS KAKENHI Grant Number 16K05257.

\end{document}